\documentclass[11pt]{amsart}
\usepackage{amssymb}
\usepackage{amsmath} 
\usepackage{amsthm} 
\usepackage{epsfig}
\usepackage{graphicx}
\usepackage{color}
\usepackage{transparent}
\usepackage{soul}
\usepackage{subfig}
\usepackage{wrapfig}
\usepackage[colorinlistoftodos]{todonotes}
 
\numberwithin{equation}{section}

\newcommand{\ve}{\mathbf{e}}

\newcommand{\vs}{\mathbf{s}}

\newtheorem*{thma}{Theorem A}

\newtheorem*{thmb}{Theorem B}

\newtheorem*{thmc}{Theorem C}

\newtheorem*{rema}{Remark}

\newtheorem{thm}{Theorem}[section]
\newtheorem{lem}[thm]{Lemma}

\newtheorem{rem}[thm]{Remark}

\newtheorem{prop}[thm]{Proposition}

\title{An Automorphism group of a rational surface: \break  Not too big not too small} 

\author{Kyounghee Kim}
\address{Department of Mathematics\\
         Florida State University\\
         Tallahassee, FL 32308}
\email{kim@math.fsu.edu}


\begin{document}

\maketitle

\begin{abstract}
This article concerns the realization problem of subgroups of Coxeter groups. 
We construct a subgroup $G$ of the Coxeter group $W_{15}$ such that $G$ is realized as automorphism groups of a rational surface $X$. The automorphism group $Aut(X)$ has six generators with positive entropy and $G \cong Aut(X)^* \cong D_3 \rtimes \mathbb{Z}$.  We also show that there is an element $\omega$ in $W_{14}$ is not realizable. 
\end{abstract}

\section{Introduction}\label{S:intro}
Let $X$ be a rational surface over $\mathbb{C}$. If $f:X \to X$ is an automorphism on $X$, due to Nagata \cite{Nagata2} one can identify the induced linear action $f^*:Pic(X) \to Pic(X)$ with an element of a Coxeter group $W_X$, which is generated by special reflections in the group of isometry on a lattice of signature $(1,n)$. With this identification, let us denote $Aut(X)^* \subset W_X$ be the group of induced actions $f^*: Pic(X) \to Pic(X) $ of $f \in Aut(X)$. 

We say an element $\omega$ of a Coxeter group is realized by a rational surface automorphism $f:X \to X$ if $\omega= f^*$ and also a subgroup $G$ of a Coxeter group is \textit{realized} by $Aut(X)$ if there is a rational surface $X$ such that $G=Aut(X)^* $.

\vspace{1ex}
It is natural to ask which element $\omega \in W_X$ can be realized by an automorphism on a rational surface. If $\omega$ is realized by a rational automorphism $f:X \to X$, then the cyclic subgroup $\langle \omega \rangle$ is realized by a group of automorphisms on a rational surface $X$. Thus, a more interesting question asked in \cite{Dolgachev-Ortland, McMullen:2007,Cantat:2018} must be identifying a subgroup of a Coxeter group realized by an automorphism group of a rational surface. A rational surface $X$ with an infinite $Aut(X)^*$ is quite special. To have such a rational surface, one needs to carefully pick a finite set of points $P= \{ p_1, \dots, p_n\}, n\ge 9$ such that $X$ is obtained by successive blowup of points in $P$. If $Aut(X)^*$ has a finite index in $W_X$, then the automorphism group $Aut(X)$ is \textit{large}, and we say the base locus $P$ of $X$ is \textit{Cremona Special}. Cantat and Dolgachev \cite{Cantat-Dolgachev} showed that there are only two possible Cremona Special sets in $\mathbf{P}^2$ over $\mathbb{C}$: either (1) $P$ is a Halphen set if $|P|=9$ or (2) $P$ is a Coble set if $|P|=10$.  

If $f \in Aut(X)$ for some rational surface $X$ with infinite order, then the cyclic group $\langle f^* \rangle = \{ (f^*)^n, n\in \mathbb{Z} \} \subset W_X$. Assuming that $f\in Aut(X)$ preserves an anticanonical cuspidal cubic, McMullen \cite{McMullen:2007} showed that the cyclic group $\langle f \rangle$ is a finite index subgroup of the full group $Aut(X)$. 

\vspace{1ex}

For a given element of the Coxeter group, finding a rational surface with an automorphism with infinite order is quite delicate. Several methods (mostly with an assumption of the existence of an invariant curve) to construct such rational surfaces have been studied in \cite{Bedford-Kim:2006, Bedford-Kim:2009, Blanc:2008, Diller:2011, McMullen:2007,Uehara:2010}. Yet there is more to be discovered. Due to the rareness of such rational surfaces, it is interesting to have an infinite non-cyclic automorphism group.

\vspace{1ex}
In this article, we present a rational surface whose automorphism group is equivalent to the direct product of $D_3$ and $\mathbb{Z}$. Furthermore, we show that the morphism from $Aut(X)$ to $Aut(X)^* \subset W_X$ is, in fact, an isomorphism. Another interesting example of a realizable non-cyclic subgroup of the Coxeter group was constructed by Blanc \cite{Blanc:2008} using involutions with a smooth cubic curve of fixed points. In Blanc's construction, the morphism from $Aut(X)$ to $W_X$ is not one-to-one.

\begin{thma}\label{thma} There is a rational surface $X$ such that its automorphism group $\text{Aut}(X)$ is generated by $6$ distinct quadratic automoprhisms with positive entropy and $\text{Aut}(X) \cong D_3 \rtimes \mathbb{Z}$.
\end{thma}

\begin{thmb}\label{thmb} There is a non-cyclic subgroup $G\subset W_{15}$ generated by six elements such that $G$ is realized by automorphisms on a rational surface $X$ in Theorem A, that is, $G \cong Aut(X)^*$. Furthermore  $G$ is isomorphic to the automorphism group $Aut(X)$ \[G \cong \text{Aut} \, (X) \]
\end{thmb}

The essentially same construction as in Theorem A and Theorem B would give more infinite subgroups with multiple generators, which are realized by automorphisms on rational surfaces. 
\begin{rema} 
For each odd integer $n \ge 5$ such that $n\not \equiv 0 $ (mod $3$), there is a subgroup $G$ in the Coxeter group $W_{3 n}$ such that $G$ is realized by automorphisms on a rational surface $X_{3n}$ and $G \cong \text{Aut}(X_{3n}) \cong  D_3 \rtimes \mathbb{Z}$. The essentially same construction as in Theorem \ref{T:autgroup} gives the corresponding rational surface $X_{3n}$ and six quadratic automorphisms generating $\text{Aut}(X_{3n})$. 
\end{rema}

\begin{rema} 
For each odd integer $n \ge 5$ such that $n \equiv 0 $ (mod $3$), there is a subgroup $G$ in the Coxeter group $W_{3 n}$ such that $G$ is realized by automorphisms on a rational surface $X_{3n}$ and $G \cong \text{Aut}(X_{3n}) \cong  (\mathbb{Z}/2 \mathbb{Z})^2 \rtimes \mathbb{Z}$. In this case, the automorphism group is generated by $3$ quadratic automorphisms. Because $n$ is divisible by $3$, an automorphism on $X_{3n}$ can not cyclically permute three irreducible components of the invariant cubic. \end{rema} 

In our construction, the base points are in the non-singular part of the anti-canonical curve, which consists of three lines joining at a single point. The non-cyclic normal subgroup in the Automorphism group corresponds to rotations of three lines in the anti-canonical curve. If $n$ is even, we can not swap two lines. (We will explain the reason in Lemma \ref{L:orbitdata}.) However, we can cyclicly permute three lines.

\begin{rema}
If $n$ is even and is not divisible by $3$, then there is a subgroup $G$ in the Coxeter group $W_{3 n}$ such that $G$ is realized by automorphisms on a rational surface $X_{3n}$ and $G \cong \text{Aut}(X_{3n}) \cong  \mathbb{Z}/3 \mathbb{Z} \rtimes \mathbb{Z}$.
\end{rema}

With iterated blowups, one can create any number of fixed elements under the induced action on the Picard Group. Thus, the number of $p$-cycle is not an obstruction for realization. If an automorphism $f:X\to X$ realizes an element $\omega$ with a period $p$-cycle of $-1$ curves, then $f$ must be a lift of a birational map with a periodic point with period $p$. Using this idea, the following Theorem shows that not every element in $\cup_{n\ge 10} W_{n}$ is realizable.

\begin{thmc} There is an element $\omega \in W_{14}$ such that $\omega$ is not realizable and thus the cyclic subgroup $\langle \omega \rangle \subset W_{14}$ is not realizable as $Aut(X)^*$ for a rational surface $X$.
\end{thmc}

The unrealizable element $\omega \in W_{14}$ in Theorem C is not an essential element in the sense that $\omega$ is conjugate into a subgroup $W_{10} \subset W_{14}$. If a non-essential element is realized by a rational surface automorphism $f: X\to X$, then the surface $X$ is non-minimal. Thus, the more interesting/ correct question should be the following: ``Is every essential element in $\cup_{n\ge 10} W_{n}$ realizable?"

\vspace{1ex}

This article is organized as follows: Section \ref{S:birational} has basic background about birational maps focusing on quadratic maps. Section \ref{S:Coxeter} discusses the connection to rational surface automorphisms and proves Theorem C. In Section \ref{S:salem}, we discuss properties of Salem numbers, which are useful to identify the Automorphism group of the rational surface we construct. Sections \ref{S:qauto} - \ref{S:surface} provide the construction of the rational surface in Theorem A-B and describe six quadratic automorphisms constructed using the Diller's method in \cite{Diller:2011}. In Section \ref{S:autgroup}, we discuss properties of automorphisms on the rational surface constructed in the previous Sections and prove Theorem A. The proof of Theorem B is given in Section \ref{S:wsub}.

\subsection*{Acknowledgement}
The author thanks E. Hironaka for interesting discussions regarding Salem numbers. The author also thanks the referee for a thorough reading and thoughtful comments and suggestions.

\subsection*{Statements and Declarations}
This material is partly based upon work supported by the National Science Foundation under DMS-1928930 while the author participated in a program hosted by the Mathematical Sciences Research Institute in Berkeley, California, during Spring 2022 semester.

\section{Quadratic birational maps}\label{S:birational}
Due to Nagata \cite{Nagata, Nagata2, Harbourne:1987}, any automorphism $f:X \to X$ on a rational surface $X$ with positive entropy is a lift of a birational map on $\mathbf{P}^2$. There are three basic Cremona transformations:
\begin{equation*}
\begin{aligned}
&J_3: [x_1:x_2:x_3] \to [x_2 x_3:x_1 x_3:x_1 x_2]\\
&J_2 : [x_1:x_2:x_3] \to [x_1 x_3:x_2 x_3:x_1^2] \\
&J_1 : [x_1:x_2:x_3] \to [x_1^2:x_1 x_2:x_2^2-x_1 x_3].\\
\end{aligned}
\end{equation*} 
Each transformation $J_i$ is a quadratic involution with an exceptional locus consisting of $i$ distinct irreducible components. It is well-known \cite{Cerveau:2008} that a birational map on $\mathbf{P}^2$ can be written as a composition of automorphisms on $\mathbf{P}^2$ and three involutions $J_i, i=1,2,3$. In fact, any quadratic birational map can be obtained from one of three involutions, $J_i$ by pre- and post- composing automorphisms on $\mathbf{P}^2$. For a birational map $f$ with the indeterminacy locus $\mathcal{I}(f)$ and any curve $V$, let us use $f(V)$ for the set-theoretic strict transformation $f(V\setminus \mathcal{I}(f))$.

\subsection{The Cremona Involution}
The Cremona involution $J_3$ has three distinct exceptional lines $E_i := \{ x_i =0 \}$ and three points of indeterminacy $e_1:=[1:0:0], e_2:=[0:1:0]$ and $e_3:=[0:0:1]$. The map $J_3$ is locally bi-holomorphic on the complement of the union of exceptional lines and maps $E_i \setminus\{ e_1,e_2,e_3\}$ to a point $e_i$ for each $i=1,2,3$ \[ J_3: E_i \setminus \{ e_1, e_2, e_3 \} \mapsto e_i \qquad \text{ for all } i = 1,2,3. \]
Since the map $J_3$ preserves a pencil of lines joining $e_i$ and a point in $E_i$, the Cremona involution $J_3$ lifts to an automorphism on a rational surface $X$ obtained by blowing up $\mathbf{P}^2$ along a set of three points $\{e_1,e_2,e_3\}$. Let $\mathcal{E}_i$ be the irreducible exceptional curve over $e_i, i=1,2,3$ and $[\mathcal{E}_0]$ be the class of the strict transformation of a generic line in $\mathbf{P}^2$. The ordered set $\{ [\mathcal{E}_0],[\mathcal{E}_1],[\mathcal{E}_2],[\mathcal{E}_3]\}$ forms a geometric basis (See \cite{dolgachev2008reflection} for the definition.) for $Pic(X)$, in other words  $\{  [\mathcal{E}_0],[\mathcal{E}_i], i=1, \dots n \}$ is a basis such that \[  [\mathcal{E}_0] \cdot  [\mathcal{E}_0] =1,\ \ [\mathcal{E}_i] \cdot  [\mathcal{E}_i] =-1  \ \text{for } i=1, \dots n,\ \ \ \text{and} \ \ \  [\mathcal{E}_i] \cdot  [\mathcal{E}_j] =0 \ \text{for } i \ne j. \]
A geometric basis for a rational surface $X$ gives an isomorphism between $Pic(X)$ and a lattice $\mathbb{Z}^{1,n}$ with signature $(1,n)$. (See Section \ref{S:Coxeter}.)

\subsection{Two more Cremona Involutions}
The involution $J_2$ has two distinct exceptional lines $E_1,E_3$ and two points of indeterminacy $e_2,e_3$. 
\[ J_2: E_1 \setminus\{e_2,e_3\} \mapsto e_2, \qquad \text{and} \qquad E_3 \setminus\{e_2,e_3\} \mapsto e_3. \] Unlike $J_3$, this involution $J_2$ requires iterated blowups to lift to an automorphism: first blowing up two points $e_2, e_3$, then a point $e_2' = \mathcal{E}_2 \cap \{ x_3=0\}$ where $\mathcal{E}_2$ is the irreducible exceptional curve over $e_2$. Let $X$ be the resulting rational surface then the ordered set $\{ [\mathcal{E}_0],[\mathcal{E}'_2],[\mathcal{E}_2+\mathcal{E}'_2],[\mathcal{E}_3]\}$ forms a geometric basis for $Pic(X)$ where $\mathcal{E}'_2$ is the exceptional curve over $e_2'$. 

\vspace{1ex}
The exceptional locus of the involution $J_1$ consists of one irreducible line $E_1$ with multiplicity $3$, that is, one needs to blow up a point three times to resolve the indeterminacy. \[ J_1: E_1 \setminus \{ e_3\} \mapsto e_3. \] To resolve the indeterminacy, one needs to blow up the point $e_3$, $e_3':=\mathcal{E}_3 \cap \{ x_1=0\}$, and $e_3'':=\mathcal{E}_3' \cap \{ x_1^2= x_2\}$ where $\mathcal{E}_3$, $\mathcal{E}_3'$ and $\mathcal{E}_3''$ are irreducible exceptional curves over $e_3$, $e_3'$ and $e_3''$ (respectively). Again, let $X$ be the resulting blowup space. In this case a geometric basis for $Pic(X)$ is given by $\{ [\mathcal{E}_0],[\mathcal{E}''_3],[\mathcal{E}'_3+\mathcal{E}''_3],[\mathcal{E}_3+\mathcal{E}'_3+\mathcal{E}''_3]\}$.

\begin{rem}
With the given geometric basis above, the lifts of all three involutions $J_i, i=1,2,3$ share the same linear action on $Pic(X)$. 
\end{rem}

\subsection{Quadratic Birational Maps}
It is known \cite{Cerveau:2008} that every quadratic birational map on $\mathbf{P}^2$ can be written as $T^- \circ J_i \circ (T^+)^{-1}, i=1,2,3$ with $T^\pm \in \text{Aut}\,(\mathbf{P}^2)$. In fact, two Cremona involutions $J_1$ and $J_2$ can be written as compositions of linear maps and $J_3$. 
We say a birational map $ \check f$ is \textit{basic} if $ \check f=T^- \circ J_3 \circ (T^+)^{-1}$.  A basic birational map $\check f= T^- \circ J_3 \circ (T^+)^{-1}$ has three distinct exceptional lines $E^+_i$ and the indeterminacy locus $\mathcal{I} = \{ p_i^+, i=1,2,3 \}$ where \[ E_i^+ = T^+ \{ x_i = 0 \}, \qquad \text{and} \qquad p_i^+ = T^- e_i.\] Similarly, we have three exceptional lines $E_i^-$ and three points of indeterminacy $p_i^-$ for $f^{-1}$.
Let us denote $\check f(V)$ the strict transform, $\check f (V \setminus\mathcal{I} ( \check f))$, of $V$ under $ \check f$. We have
\[ \check f : E_i^+ \mapsto p_i^-, \qquad\text{and} \qquad \check f^{-1} : E_i^- \mapsto p_i^+ \qquad \text{for } i =1,2,3.\] Note that any non-basic quadratic birational map is obtained from either $J_2$ or $J_1$, and thus the number of irreducible components of its exceptional locus is strictly less than $3$. 

\subsection{Dynamical Degree} Suppose $g:X \dasharrow X$ is a birational map on a compact complex manifold. In this case, the Picard group $Pic(X)$ is identical to the N\'{e}ron-Severi group $NS(X)$. There is a natural induced linear map $g^* : Pic(X) \to Pic(X)$. The dynamical degree $\lambda(g)$ is birational invariant and given by 
\[ \lambda(g) : = \lim_{n \to \infty} ||(g^n)^*||^{1/n}. \] 
We say $g$ is \textit{algebraically stable} if the induced linear action $g^*: Pic(X) \to Pic(X)$ satisfies the following condition
 $ (g^*)^n \ =\ (g^n)^*$ for all $\ n \ge 1$. If $g$ is algebraically stable, the dynamical degree is given by the spectral radius of the linear action $g^*$.

In general, a birational map $\check f: \mathbf{P}^2 \dasharrow \mathbf{P}^2$ does not satisfy the algebraic stability condition. However, Diller and Favre \cite{Diller-Favre:2001} showed that one can always construct an algebraically stable modification, that is, there is a blowup $X= B\ell_P \mathbf{P}^2$ along a finite set $P$ of points such that the lift $f:X\dasharrow X$ is algebraic stable and the dynamical degree of $\check f$ is given by $\lambda(\check f) = \lambda(f)$, the spectral radius of the induced linear action $f^*$ on $Pic(X)$. Blanc and Cantat \cite{blancdynamical} showed that a birational map $\check f$ lifts to an automorphism if and only if the dynamical degree is a Salem number.

\subsection{Orbit Data}
Suppose $\check f$ is basic. For each exceptional line $E_i^+$, we define $n_i \in \mathbb{N} \cup \{ \infty\}$ by 
\[ \left\{ \begin{aligned} & n_i \ = \ \min \{ n: \text{dim} \check f^n (E_i^+) \lneq \text{dim} \check f^{n+1} (E_i^+) \},\\ & n_i \ =\ \infty \qquad \text{ if } \ \ \text{dim} \check f^n (E_i^+) = 0 \ \ \text{for all } n\ge 1 \end{aligned} \right. \]
If $n_i < \infty$, the last point of the orbit $\check f^{n_i} E_i^+ \in \mathcal{I}^+ = \{ p_1^+,p_2^+,p_3^+\}$. Thus we can also define a permutation $\sigma \in \Sigma_3$ such that 
\[ \check f^{n_i} E_i^+ = p_{\sigma(i)}^+.\]
We call this numerical information $n_1, n_2, n_3$ and a permutation $\sigma$ the orbit data of the basic birational map $ f$. 

\begin{rem}For a non-basic quadratic birational map, we count the orbit length with multiplicity with the identity permutation. For example, if $\check f = T^- \circ J_1 \circ (T^+)^{-1}$, we let $E_1^+ = T^+\{x_1=0\}$ and the orbit data are $n_1, n_1,n_1$ and the identity permutation. 
\end{rem}

\begin{lem}
If $\check f: \mathbf{P}^2 \dasharrow \mathbf{P}^2$ is a basic quadratic birational map, then so is the inverse $\check f^{-1}$.
Furthermore If $\check f$ has orbit data $n_1,n_2,n_3$ with a permutation $\sigma \in \Sigma_3$, then the orbit data for $\check f^{-1}$ are given by $n_{\sigma^{-1}(1)},n_{\sigma^{-1}(2)},n_{\sigma^{-1}(3)}$ with $\sigma^{-1}$.
\end{lem}

\begin{proof} 
Since every quadratic map $\check f$ is given by a composition $T^- \circ J_i \circ (T^+)^{-1}$ for some $i=1,2,3$, it is obvious. 
\end{proof}

\subsection{Lifting to an Automorphism}\label{SS:basis}
Suppose the orbit data of a basic birational map $\check f$ 
consists of three positive integers $n_1,n_2,n_3 < \infty$ and let $P_i= \{ p_{i,j} := \check f^j E_i^+, \ 1 \le j \le n_i \}$. Then $\check f$ lifts to an automorphism on a blowup 
 \[ \pi : X = X_3 \xrightarrow{\pi_3} X_{2} \xrightarrow{\pi_{2}} X_1 \xrightarrow{\pi_1} X_0 = \mathbf{P}^2 \] 
where $\pi_j:X_j\to X_{j-1}$ is a blowup of a set of points $P_{i_j} \subset X_{j-1}$ with $n_{i_1}\le n_{i_2} \le n_{i_3}$.

If a quadratic birational map $\check f = T^- \circ J_2 \circ (T^+)^{-1}$, then the orbit lengths are $n_1, n_1, n_3$. Let $P_i= \{ p_{i,j} := \check f^j E_i^+, \ 1 \le j \le n_i \}$ for $i=1,3$. When both $n_1, n_3 < \infty$, $\check f$ lifts to an automorphism on a rational surface
 \[ \pi : X = X_{2} \xrightarrow{\pi_{2}} X_1 \xrightarrow{\pi_1} X_0 = \mathbf{P}^2 \] 
where 
\begin{itemize}
\item $\pi_1$ is a successive blowup of two set of points $P_1$ and $P_3$ in the increasing order of $n_i$, and 
\item $\pi_2$ is a blowup of a set of points $P_2 : = \{ p_{2,j} := \check f_1^j E_1^+, \ 1 \le j \le n_1 \} \subset X_1$ where $f_1:X_1 \to X_1$ is the induced biratioanl map. 
\end{itemize}

Lastly, if a quadratic birational map $\check f = T^- \circ J_1 \circ (T^+)^{-1}$ with $n_1<\infty$, then the rational surface $\pi:X\to \mathbf{P}^2$ is constructed by successive blowups along a set of poitns $P_1 = \{ p_{1,j} := \check f^j E_1^+, \ 1 \le j \le n_1\}$, $P_2 : = \{ p_{2,j} := \check f_1^j E_1^+, \ 1 \le j \le n_1 \} \subset X_1$, and $P_3 : = \{ p_{3,j} := \check f_2^j E_1^+, \ 1 \le j \le n_1 \} \subset X_2$. 

\vspace{1ex}
In all three cases, a geometric basis of $Pic(X)$ is given by the classes of total transformations of exceptional curves over a point in $P= \cup P_i$ together with the class of a generic line.

\section{Coxeter group $W_n$}
\label{S:Coxeter}
Let $\mathbb{Z}^{1,n}$ be the unimodular lattice of signature $(1,n)$ with a basis $( \ve_0,\ve_1,\dots,\ve_n)$ such that \[ \ve_0^2 = 1, \quad \ve_i^2 = -1 \quad \text{ for } i \ge 1, \quad \text{ and } \quad \ve_i \cdot \ve_j = 0 \quad \text{for } i \ne j. \] Suppose $n \ge 3$. Let \[ \kappa_n = -3 \ve_0 + \ve_1 + \ve_2 + \cdots +\ve_n, \qquad \text{and} \] \[ \alpha_0 = \ve_0 - \ve_1- \ve_2- \ve_3,\quad \alpha_i = \ve_i - \ve_{i+1} \ \ \text{ for } i \ge 1. \]
The sublattice $L_n$ with a basis $(\alpha_0, \alpha_1, \dots, \alpha_{n-1})$ is the orthogonal complement of the vector $\kappa_n$. 
Since $\alpha_i^2 = -2$ for all $i \ge 0$, we can define the reflection $\vs_i$ through the vector $\alpha_i$ by \[ \vs_i : x \mapsto x + (x \cdot \alpha_i) \alpha_i. \]
These reflections generate the Coxeter group $W_n \subset O(L_n)$ with respect to the graph $E_n$ in Figure \ref{fig:coxetergraph}. Let $m_{i,i}=1$, and $m_{i,j} = 3 $ if $\vs_i$ and $\vs_j$ are joined by an edge, then we have
\[ W_n = \langle \vs_0,\vs_1,\dots, \vs_{n-1} \,|\, (\vs_i \vs_j)^{m_{i,j} }=1 \rangle. \]
 The subgroup $S_n$ generated by $\{ \vs_1,\vs_2, \dots, \vs_{n-1}\} \subset W_n $ is isomorphic to a permutation group on $n$ elements. Each element in $S_n$ permutes the basis elements $\mathbf{e}_1, \dots, \mathbf{e}_n$. Let $\vs \in S_n$ such that $\vs$ interchanges $\mathbf{e}_1\leftrightarrow \mathbf{e}_i$, $\mathbf{e}_2 \leftrightarrow \mathbf{e}_j$, $\mathbf{e}_3 \leftrightarrow \mathbf{e}_k$ for some distinct indcies $i,j,k$, and fixes everything else. Then we have \[ \vs_{ijk} = \vs^{-1} \circ \vs_0 \circ \vs \in W_n.\] This element $\vs_{ijk} \in W_n$ acts on $\mathbb{Z}^{1,n}$ as a reflection through the vector $\mathbf{e}_0 - \mathbf{e}_i-\mathbf{e}_j-\mathbf{e}_k$.

\begin{figure}
\centering
\includegraphics{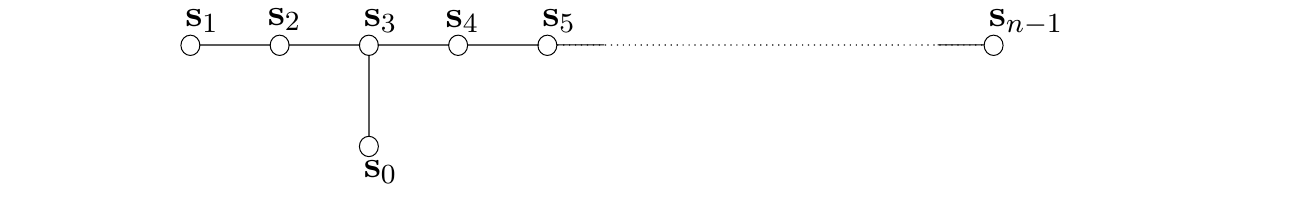}
\caption{Coxeter Graph $E_n$ \label{fig:coxetergraph} }
\end{figure} 

\subsection{The smallest spectral radius in $W_n$.}
Let $\rho(\omega)$ be the spectral radius of this linear map $\omega \in W_n$. 
We say $\omega \in W_n$ is \textit{essential} if $\omega $ is not conjugate into a subgroup $W_I$ generated by a subset  $\{ \vs_i, i \in I\} \subset \{ \vs_0, \vs_1, \dots, \vs_{n-1} \}$ for an index set $I \subset \{0,1, \dots, n-1\}$: 
\[ W_I \ =\  \langle  s_i, i \in I | (s_j s_k)^{m_{j,k}} = 1 \rangle \]

\begin{thm}\cite{McMullen:2002} 
If $\omega \in W_n$ is essential, then the spectral radius of $\omega$ is greater than or equal to the spectral radius of the Coxeter element $\omega_n\in W_n$ \[ \rho(\omega) \ge \rho(\omega_{n}). \]
\end{thm}
The spectral radius of the Coxeter element of $W_{10}$ is given by Lehmer's number, $\rho(\omega_n)$ is strictly increasing in $n$, and the spectral radius of any element in $W_n, n\le 9$ is equal to $1$. Thus we have
\begin{thm}\cite{McMullen:2002} \label{C:smallestSalem}
Suppose $\omega \in \bigcup_{n\ge 10}W_n$. If the spectral radius $\rho(\omega)>1$ then $\rho(\omega) \ge \lambda_{\star} \approx 1.17628$ where $\lambda_\star$ is the largest real root of the Lehmer's polynomial \[ t^{10}+t^9-t^7-t^6-t^5-t^4-t^3+t+1. \]
\end{thm}

\subsection{Connection to Rational surface automorphisms}

A rational surface $X$ is given by successive blowups $\pi:X \to \mathbf{P}^2$ with a factorization \[ \pi : X = X_n \xrightarrow{\pi_n} X_{n-1} \xrightarrow{\pi_{n-1}} X_{n-2} \rightarrow \cdots \rightarrow X_1 \xrightarrow{\pi_1} X_0 = \mathbf{P}^2 \] where $\pi_i : X_i \to X_{i-1}$ is a blowup of a point $x_i \in X_{i-1}$. Let $E_i \subset X_i$ denote an exceptional curve lying over $x_i \in X_{i-1}$ and let $\mathcal{E}_i = ( \pi_{i+1} \circ \dots \pi_n)^{-1} E_i$ be the total transform of $E_i$. If $x_i $ is an infinitely near point in $\mathbf{P}^2$, i.e. $(\pi_{1} \circ \dots \circ \pi_{i-1}) \, x_i = (\pi_{1} \circ \dots \circ \pi_{j-1}) \,x_j$ for some $j \lneq i$, then $\mathcal{E}_i$ is a reducible curve. Let $e_0 = \pi^*([\ell])$ where $\ell$ is a generic line in $\mathbf{P}^2$ and $e_i = [\mathcal{E}_i]$ be the cohomology class of a curve $\mathcal{E}_i$ for $i = 1, 2, \dots, n$. These classes $e_0,e_1, \dots, e_n$ give a geometric basis of the Picard lattice of $X$. With the intersection product, one can check that \[ e_0 \cdot e_0 = 1, \quad e_i \cdot e_i = -1, \text{ for } i =1, \dots, n, \quad \text{and } e_i \cdot e_j =0, \text{ for }i \ne j. \]
Thus, there is a natural isomorphism 
\[ \phi : \mathbb{Z}^{1,n} \to Pic(X) \] such that \[ \phi(\mathbf{e}_i) = e_i,\quad i = 0, \dots n \] Due to Nagata \cite{Nagata2,Dolgachev-Ortland}, if $f:X \to X$ is an automorphism then there is an element $\omega \in W_n$ such that for a linear action $f^*$ on $Pic(X)$ \begin{equation}\label{E:iden} \phi\circ \omega = f^* \circ \phi. \end{equation} In this case, we say $\omega$ is \textit{realized} by a rational surface automorphism $f :X \to X$. Let us identify $f^* = \omega$ through the isomorphism $\phi$.

\subsection{Not-realizable element}
Let us consider $\omega \in W_{16}$ given by $\omega = \omega_1 \cdot \omega_2$ where 
\begin{equation}\label{E:norealization}
\omega_1= \vs_0 \vs_1 \cdots \vs_9 \in W_{n\ge 10}, \quad \text{and} \quad \omega_2 = \vs_{11} \vs_{12} \vs_{13} \in W_{14}. 
\end{equation}
The action $\omega$ on $\mathbb{Z}^{1,14}$ is given by  
\begin{equation}\label{E:nAction}
\begin{aligned}
\omega\ :\ \ & \ve_0 \mapsto 2 \ve_0- \ve_1-\ve_2-\ve_3, \\
& \ve_1 \mapsto \ve_0 - \ve_1-\ve_3,\ \ \ \ve_2 \mapsto \ve_0 - \ve_1 - \ve_2,\\
& \ve_3 \mapsto \ve_4 \mapsto \cdots \mapsto \ve_{10} \mapsto \ve_0 - \ve_2-\ve_3,\\
&\ve_{11} \mapsto \ve_{12}\mapsto \ve_{13} \mapsto \ve_{14} \mapsto \ve_{11}.\\
\end{aligned}
\end{equation}

\begin{lem}
If $\omega$ is realized by an automorphism $f:X \to X$ on a rational surface $X$, then there is a rational surface $X_1$ obtained by blowing down four $(-1)$-curves $\mathcal{E}_i, i=11, \dots, 14$ with $e_i = [\mathcal{E}_i]$ to $4$ distinct points in $X_1$ and the induced map $f_1:X_1 \to X_1$ realizes $\omega_1 \in W_{10}$ and there is a period $4$ cycle of points. 
\end{lem}

\begin{proof} The identification between $f^*$ and $\omega$ comes from a choice of a geometric basis. Since $f$ is quadratic, due to \cite[Theorem~5.2]{dolgachev2008reflection} we may choose the geometric basis as defined in Section \ref{SS:basis}. Except for the class of generic line, all other geometric basis elements correspond to $(-1)$ curves. Thus, we can blowdown four $(-1)$ curves $\mathcal{E}_i, i=11, \dots, 14$ to $4$ distinct points. 
\end{proof}

\begin{lem}
If a rational surface automorphism $f_1: X_1 \to X_1$ realizes $\omega_1 \in W_{10}$, then $f_1$ does not have a (non-fixed, no period 2) period $4$ cycle.
\end{lem}

\begin{proof}
Suppose there is a curve $C$ of fixed points for $f^n_1$ for some $n\ge 2$. The characteristic polynomial of $\omega_1$ is the product of $(t-1)$ and the irreducible Lehmer's Polynomial. Since no root of Lehmer's Polynomial is a root of unity and the $[C]$ an eigenvector of $\omega_1^4$ with the eigenvalue $1$, it follows that the curve $C$ is, in fact, fixed. Thus all periodic (non-fixed) points of $f_1$ are isolated. Since we have $(f_1^*) ^n = (f_1^n)^*$ and $\text{Tr}\, \omega_1^2 = \text{Tr}\, \omega_1^4$ , from Lefschetz's fixed point formula we see that there is no period $4$ cycle. \end{proof}

\begin{proof}[Proof of Theorem C]
Theorem C is the immediate consequence of the previous two Lemmas.
\end{proof}

McMullen in \cite{McMullen:2007} pointed out that $\vs_0 \in W_8$ can not be realized with a rational surface automorphism $f:X\to X$ if $X$ is a blowup of $\mathbf{P}^2$ along a set of $8$ distinct points in $\mathbf{P}^2$. On the other hand, one can perform iterated blowups over the fixed points. Let $X_1$ be the blowup of $\mathbf{P}^2$ along the set of four points $x_1=[1:0:0], x_2=[0:1:0],x_3=[0:0:1]$ and $x_4=[1:1:1]$ and let $e_i$ be the class of exceptional curves over $x_i$. Then the Cremona involution $J_3$ lifts to an automorphism $f_1:X_1 \to X_1$ with $f_1^* = \vs_0 \in W_4$, that is $f_1^* e_4= e_4$. 
Since the exceptional curve over $x_4$ is an invariant curve for $f_1$, there is a fixed point $x_5$ in the exceptional curve. Thus $f_1$ lift to another automorphism on a rational surface realizing $\vs_0 \in W_5$. By repeating this procedure, we can realize $\vs_0 \in W_n$ for all $n \ge 4$.

\section{Salem Numbers}\label{S:salem}
A Salem number $\tau$ is an algebraic integer such that the minimal polynomial is a reciprocal polynomial of degree $\ge 4$ such that there are exactly two real roots outside the unit circle and all other roots are (non-real) complex with modulus $1$. If an element $\omega \in W_n$ of a Coxeter group has the spectral radius $>1$, then the spectral radius is a Salem number. Since the induced action $f^*|_{Pic(X)}$ of a rational surface automorphism $f:X \to X$ can be identified with an element in a Coxeter group, Salem numbers play an essential role. Blanc and Cantat \cite{blancdynamical} showed that a birational map $\check f$ on $\mathbf{P}^2$ lifts to a rational surface automorphism with positive entropy if and only if the dynamical degree $\delta(\check f)$ is a Salem number. 
 In this section, we present the useful results on Salem numbers. The following results of Salem are discussed and proved in the survey paper by Smyth \cite{Smyth:2015}. 

\begin{prop}\cite{Salem:1945} \label{P:Salem1}
If $\tau$ is a Salem number of degree $d$, then so is $\tau^n$ for all $ n \in \mathbb{N}$. 
\end{prop}

The following proposition by Salem gives a specific idea about the number fields containing Salem numbers. The last two statements are specifically useful for our results. 
\begin{prop}\cite{Salem:1945}\label{P:Salem2}
\begin{enumerate}
\item A number field $K$ is of the form $\mathbb{Q}(\tau)$ for some Salem number $\tau$ if and only if $K$ has a totally real subfield $\mathbb{Q}(\alpha)$ of index $2$, and $K= \mathbb{Q}(\tau)$ with $\tau + \tau^{-1} = \alpha$, where $\alpha>2$ is an irrational algebraic integer, all of whose conjugates not equal to $\alpha$ lie in $(-2,2)$.  
\item Suppose that $\tau$ and $\tau'$ are Salem numbers with $\tau' \in \mathbb{Q}(\tau)$. Then $\mathbb{Q}(\tau') = \mathbb{Q}(\tau)$ and, if $\tau'>\tau$, then $\tau'/\tau$ is also a Salem number in $\mathbb{Q}(\tau)$.
\item If $K= \mathbb{Q}(\tau)$ for some Salem number $\tau$, then there is a Salem number $\tau_1 \in K$ such that the set of Salem numbers in $K$ consists of the power of $\tau_1$. 
\end{enumerate}
\end{prop}

I would like to thank Eko Hironaka for helping me with the following proposition. 
\begin{prop}\label{P:prodSalem}
Suppose $\alpha \ne \beta$ are two distinct Salem numbers. Then one and only one of the the followings occurs
\begin{enumerate}
\item there is a Salem number $\tau$ such that $\alpha = \tau^m$ and $\beta = \tau^n$ for some $m,n \ge 1$
\item the products of $\alpha, \beta$ and their reciprocals, ($\alpha\beta, \alpha/\beta, \beta/\alpha,$ and $ 1/(\alpha\beta)$) are not Salem numbers.
\end{enumerate}
\end{prop}
\begin{proof}
Suppose there is no Salem number $\tau$ such that $\alpha = \tau^m$ and $\beta = \tau^n$ for some $m,n \ge 1$. Let $K$ be the splitting field of $\alpha$ over $\mathbb{Q}$ and let $H = \text{Gal}(K/\mathbb{Q})$ be the Galois group of $K$ over $\mathbb{Q}$. Then $\beta \not\in K$. Now consider the smallest Galois extension $L$ over $\mathbb{Q}$ which contains both Salem numbers $\alpha$ and $\beta$ and let $G = \text{Gal}(L/\mathbb{Q}$). Since $K \subset L$ is a splitting field, by the fundamental Theorem of Galois theory, we have a normal subgroup $N \lhd G$ such that \[ N = \{ g \in G: g (a) = a , \forall a \in K \}.\] Thus every element in $N$ fixes all Galois conjugate of $\alpha$ including $\alpha$ itself. Since $\beta \not\in K$, there is $g \in N \lhd G$ such that $g(\beta) \ne \beta$, in other words, $g(\beta)$ is a Galois conjugate of $\beta$. Thus we have either $g(\beta) = 1/\beta$ or $|g(\beta)|=1$. Since $g \in N \lhd G$, we have $ g(\alpha\beta) = \alpha g(\beta)$. If $|g(\beta)|=1$, then $g(\alpha\beta) $ is a non-real complex number outside the unit circle and thus $\alpha\beta$ has a Galois conjugate outside the unit circle. If $g(\beta) = 1/\beta$, then $\alpha\beta$ is conjugate to a real number $\alpha/\beta \ne 1/(\alpha\beta)$. In either case, we see that $\alpha\beta$ is not Salem. Similarly we can conclude that $\alpha/\beta, \beta/\alpha,$ and $ 1/(\alpha\beta)$ are not Salem numbers.
\end{proof}

\section{Quadratic rational surface automorphisms}\label{S:qauto}
Searching for basic birational maps on $\mathbf{P}^2$ with given orbit data is a challenging problem since it involves solving systems of high-degree algebraic equations with a large number of variables. With the extra condition of the existence of an invariant cubic, McMullen \cite{McMullen:2007} and Diller \cite{Diller:2011} provide recipes to construct birational maps which lift to rational surface automorphisms. For a given orbit data with a specified invariant cubic, Diller's Method gives an explicit formula for the corresponding birational map. This method also works for the higher dimension \cite{Bedford-Diller-K}. In this article, we are considering basic quadratic birational maps fixing three lines joining at a single point with orbit data
\begin{equation}\label{E:555}
 n_1= n_2 = n_3 = n\ge 4, \qquad \text{and } \ \ \sigma \in S_3
\end{equation}

\subsection{Quadratic birational map fixing three concurrent lines}
Let $L_j =\{ \gamma_i (t) : t \in \mathbb{C} \cup \{ \infty\} \} \subset \mathbf{P}^2$ be the line given by the parametrizations 
\begin{equation}\label{E:param}
\gamma_1(t) = [-t:1:1],\quad \gamma_2(t) = [t:1:0],\quad \text{and} \quad \gamma_3(t) = [t:0:1] 
\end{equation}
Let $C=L_1 \cup L_2 \cup L_3$ be a reducible cubic consisting of three lines joining at a single point $[1:0:0]$. Let us fix affine coordinates $(t, i) = \gamma_i(t) $ for a point in the cubic $C$. 
\begin{equation}\label{E:cubic}
C= \{ (t,i) := \gamma_i(t) | t \in \mathbb{C} \cup \{ \infty\}, i=1,2,3\}
\end{equation}
The given parametrization in (\ref{E:param}) ensures that if three points on the cubic, $p_1 = (t_1,1), p_2= (t_2,2),$ and $p_3=(t_3,3)$ are co-linear than $t_1+t_2+t_3 = 0$. We say a birational map $\check f$ on $\mathbf{P}^2$ properly fixes $C$ in (\ref{E:cubic}) if $\check f(C) = C$ and none of the points of indeterminacy $p_i^\pm$ are the common intersection point $C_{sing}$ of $C$.

Suppose $\check f: \mathbf{P}^2 \dasharrow \mathbf{P}^2$ be a quadratic map that properly fixes $C$. For a given irreducible curve $V \subset \mathbf{P}^2$ of degree $k$, the degree of the image $\check f(V)$ is $2 k - (m_1+m_2+m_3)$ where $m_i$ is the vanishing order of the point of indeterminacy $p_i^+$ in $V$. Since the invariant cubic $C$ consists of three lines, each irreducible component $L_i$ should contain exactly one point of indeterminacy. It follows that there are three distinct points of indeterminacy, and thus we have
\begin{lem}\label{L:basic}
If a quadratic birational map $\check f: \mathbf{P}^2 \dasharrow \mathbf{P}^2$ properly fixes a concurrent three lines $C$, then $\check f$ is basic.  
\end{lem}

We may assume that $ p_i^+ \in L_i \setminus \{C_{sing}\}$ for all $i=1,2,3$.
It follows that $\check f$ permutes three lines. Suppose $\check f (L_i) = L_{\tau(i)}$ for some permutation $\tau \in S_3$. 
Since the exceptional line $E_1^+$ intersects $L_j$ at one point and since for $j=2,3$, the indeterminant point $p_j^+$ is contained in both $E_1^+$ and $L_j$, we have non-indeterminate point $q_1\in L_1 \cap E_1^+$. Similarly we have 
$ L_i \cap E_i^+ = \{ q_i\} \not\in \mathcal{I}(\check f)$ for all $i = 1,2,3$, and $p_i^- = \check f (E_i^+) \in L_{\tau(i)}$ for all $ i=1,2,3.$

 \begin{lem} \label{L:indp}
 Suppose $\check f$ is a basic quadratic map properly fixing $C = L_1 \cup L_2\cup L_3$. And suppose $\check f(L_i) = L_{\tau(i)}$ for some $\tau \in S_3$. Then after renaming it, if necessary, we have 
\[p_i^+, q_i \in L_i , \qquad \text{and } \quad p_i^- \in L_{\tau(i)} \qquad \text{ for all } i =1,2,3 \]
where $q_i$ is the unique non-indeterminate intersection point $q_i \in E^+_i \cap (C \setminus \mathcal{I})$.
 \end{lem}
 
Because of the previous Lemma \ref{L:indp}, we have few restrictions on the orbit data for a basic quadratic map $\check f$ properly fixing $C$. 
\begin{lem}\label{L:orbitdata}
Suppose $\check f$ is a basic quadratic map properly fixing $C = L_1 \cup L_2\cup L_3$ with orbit data $n_1, n_2,n_3, \sigma \in S_3$. Suppose $\check f(L_i) = L_{\tau(i)}$ for some $\tau \in S_3$. If $n_i <\infty$ for all $i$, then we have
\begin{itemize}
\item if $\tau=Id$ is identity, then so is $\sigma$. 
\item if $\tau$ permutes $i$ and $j$, then either(1) both $n_i$ and $n_j$ are odd and $\sigma= \tau$, or (2) both $n_i$ and $n_j$ are even and $\sigma= Id$.
\item if $\tau$ is a cyclic permutation, then either (1) $n_i \equiv 0$ (mod $3$) for all $i$ and $\sigma=Id$, (2) $n_i \equiv 1$ (mod $3$) for all $i$ and $\sigma = \tau$, or (3) $n_i \equiv 2$ (mod $3$) for all $i$ and $\sigma = \tau^2$.
\end{itemize}
\end{lem}

\begin{proof}
Notice that for each $i=1,2,3$, we have $q_i, p_i^+ \in L_i$ and $\check f^j(q_i) = f^j(E_i^+)$ for all $j \le n_i$. Since $\check f^j(q_i) \in L_{\tau^j(i)}$ and $\check f^{n_i}(q_i) = p_{\sigma(i)}^+ \in L_{\sigma(i)}$, we have \[ \sigma(i) = \tau^{n_i}(i) \qquad \text{for all } i=1,2,3, \] and thus the statement of this Lemma follows.
\end{proof}

Suppose all orbit lengths of a basic quadratic map are equal to five, $n_1= n_2=n_3 =5$. Then all six permutations in $S_3$ are possible candidates for basic quadratic birational maps properly fixing three concurrent lines. The explicit formulas for the minimal polynomials (up to cyclotomic factors) of the dynamical degree using orbit data are given in \cite{Bedford-Kim:2004}: With orbit lengths $n_1,n_2, n_3 \in \mathbb{N}$, we have 

\begin{itemize}\addtolength{\itemsep}{1ex}
\item a cyclic permutation :\hfill\break \[ \chi(t) \ = \ (t-1) ((t^{n_1}+1)(t^{n_2}+1)(t^{n_3}+1)+1)-(t^{n_1+n_2+n_3}-1),\]
\item a transposition $(i\, j)$ with $\{i,j,k\} = \{1,2,3\}$: \hfill\break \[\chi(t) \ =\ (t-1) (t^{n_k} (t^{n_i}+1)(t^{n_j}+1)- t^{n_i} - t^{n_j} -2)- (t^{n_i+n_j}-1) (t^{n_k}-1),\]
\item an identity permutation : \hfill\break \[ \chi(t) \ = (t-2) t^{n_1+n_2+n_3}+(t^{n_1+n_2}+t^{n_2+n_3}+ t^{n_1+n_3}) - t (t^{n_1} + t^{n_2}+t^{n_3})+ 2 t-1\]
\end{itemize}
With these formulas, we see that all six cases share the same dynamical degree given by the largest real root $\delta$ of a Salem polynomial $\chi(t)$ of degree $4$
\begin{equation}\label{E:ddegree} \chi_5(t) = t^4-2 t^3+t^2-2 t+1, \qquad \text{and } \quad \delta \approx 1.8832 \end{equation}

Furthermore, It is known that the dynamical degree is the largest real root of a Salem polynomial, which is determined by orbit data. Also, it is known  \cite{Bedford-Kim:2004} that if the dynamical degree is bigger than $1$ then the dynamical degree is strictly increasing in orbit lengths. It follows that there are at most six distinct orbit data for quadratic birational maps fixing concurrent lines with the dynamical degree $\delta \approx 1.8832 $ in (\ref{E:ddegree}).

 \subsection{Diller's Method}\label{SS:diller}
 To construct a basic map with the orbit data in (\ref{E:555}) which properly fixes three concurrent lines defined in (\ref{E:cubic}), we apply the construction by Diller in \cite{Diller:2011, Bedford-Diller-K} in which all possible invariant cubics were discussed.
 The original Theorem in \cite{Diller:2011} assumes that a basic birational map properly fixes a cusp cubic. The discussion for a basic birational map properly fixing three concurrent lines is in Section $4$ in \cite{Diller:2011}.
In \cite{Bedford-Diller-K}, the same method is applied to the basic maps in the higher dimensions. This Section only gives the necessary components for constructing a basic birational map fixing three concurrent lines. The following Theorem is slightly modified from the original versions \cite[Theorem~2.6]{Bedford-Diller-K} for the mappings properly fixing pre-determined three concurrent lines. 
 \begin{thm}\cite{Diller:2011,Bedford-Diller-K}\label{T:bdk}
 Let $C$ be the union of three lines joining at the same point defined in (\ref{E:cubic}). Suppose $a \in \mathbf{C}^*$ and $t_j^+ \in \mathbf{C}$, $j=1,2,3$, are distinct paramenters satisfying $t_1^+ + t_2^+ + t_3^+ \ne 0$. Then there is a unique basic birational map $\check f = T^- \circ J_3\circ (T^+)^{-1}$, $b \in \mathbf{C}$, and a permutation $\tau \in S_3$ such that 
 \begin{itemize} 
 \item $\check f$ properly fixes $C$ with $\check f|_C$ is given by $\check f|_C ((t, i) = (a t+b, \tau(i))$
 \item $(t_i^+, i) = T^-(e_i)$
 \item $ b = \frac{1}{3} a \sum t_i^+$
 \end{itemize}
 \end{thm}

Suppose $\check f =T^- \circ J_3 \circ (T^+)^{-1}$ is a basic birational map properly fixing $C=L_1\cup L_2 \cup L_3$ in (\ref{E:cubic}) and $\check f (L_i) = L_{\tau(i)}$ for some permutation $\tau \in S_3$. With the affine coordinates $(t,i) = \gamma_i(t)$ for $C$ and the parametrizations $\gamma_i$ defined in (\ref{E:param}), we have the restriction map on the cubic $C$ given by $ f|_C : (t, i) \mapsto ( a t +b , \tau(i) ), $ for some $ a, b \in \mathbb{C}. $ By conjugating with an automorphism on $\mathbf{P}^2$, we may assume that $f|_C(1,i) = (1, \tau(i))$, that is, \[ f|_C : (t, i) \mapsto ( a (t-1)+1 , \tau(i) ) \qquad \text{for some } a \in \mathbb{C} \setminus \{ 0 \}. \]

Based on Lemma \ref{L:indp}, let us suppose that $p_i^+, q_i  \in L_i $ for $i=1,2,3$ and thus set \begin{equation}\label{E:param} p_i^+ = \gamma(t_i^+,i), \qquad p_i^- = \gamma(t_i^-, \tau(i)) \qquad \text{and }\qquad q_i =\gamma(s_i,i)=\gamma ( f^{-1}|_C (t_i^-, \tau(i))). \end{equation} Let us also suppose that the orbit data of $\check f$ is given by (\ref{E:555}) with appropriate permutation $\sigma$ with respect to the Lemma \ref{L:orbitdata}. 
Since each $n_i = n$, we have  
\begin{equation}\label{E:grouplaw} ( t_{\sigma(i)}^+, \sigma(i))= \check f_C^{n-1} (t_i^-, \tau(i)) = (a^{n-1} (t_i^- -1) +1, \sigma(i) ) \end{equation} Since the exceptional line $E_1^+$ passes through $p_2^+, p_3^+$ and $q_1$, we have $ s_1 + t_2^++ t_3^+ =0$. Thus using (\ref{E:param}) and (\ref{E:grouplaw}), we get 
\begin{equation}\label{E:eq1}  a^{n+1} ( t^-_{\sigma^{-1}(2)}+t^-_{\sigma^{-1}(3)} -2) =1-t_1^--3 a \end{equation}
Similarly we have $ s_2 + t_1^++ t_3^+ =0$ and $ s_3 + t_1^++ t_2^+ =0$. Thus we get  
 \begin{equation}\label{E:eq23}
\begin{aligned} & a^{n+1} ( t^-_{\sigma^{-1}(1)}+t^-_{\sigma^{-1}(3)} -2) =1-t_2^--3 a \\
& a^{n+1} ( t^-_{\sigma^{-1}(1)}+t^-_{\sigma^{-1}(2)} -2) =1-t_3^--3 a \\
\end{aligned}
 \end{equation}
 On the other hand, a generic line $h$ will intersect each line $L_i$ at one point, say $x_i = \gamma(r_i, i)$ and $x_1+x_2+x_3 =0$. Furthermore, the degree of $\check f = 2$ and thus $\#| \check f (h) \cap C |= 6$. Since $h$ also intersects each exceptional line at a single point, we have $ p_i^- \in \check f(h) $ for all $i=1,2,3$. This gives another condition on the parameters for critical images $ \sum a (r_i -1) +1 + \sum_i t_i^- = a ( \sum r_i) + 3 (1-a) + \sum t_i ^- = 0$, i.e.
 \begin{equation}\label{E:gline}
 \sum_i t_i^- = 3( a-1) 
 \end{equation}
 And similarly we have $\sum_i t_i ^+ = 3 \,\frac{1-a}{a}$.
Solving $4$ equations (\ref{E:eq1}- \ref{E:gline}), we get \[ a^{n+1} - 2 a^n + 2a -1 =0\quad \text{and} \quad t_i^-= \frac{-\,3 a }{1+ 2 a^n} \ \ \text{for all } i = 1,2,3. \]

\begin{prop}\label{P:param} Suppose $n \ge 4$.  There exists a basic birational map $\check f$ properly fixing three concurrent lines with orbit lengths $n,n,n$. This birational map $\check f$ lifts to an automorphism and is given by the largest root of the polynomial 
\begin{equation}\label{E:ddeg}
\chi_n (t) = t^{n+1} - 2 t^n + 2 t -1 
\end{equation}
\end{prop}
\begin{proof}
Let $\alpha$ is a non-cyclotomic root of $\chi_n(t)$ and let
\begin{equation}\label{E:tsol} t_1\ =\  \frac{-\,3 \alpha }{1+ 2 \alpha^n}, \quad \text{and} \quad t_j = \frac{-\,3 \alpha^{j} }{1+ 2 \alpha^n} , j=2, 3, \dots, n 
\end{equation}
Then $ 3 t_n \ne 0$ and thus by Theorem \ref{T:bdk} we know there exists a basic birational map $\check f$ such that $\check f|_C: ( t, i) = ( \alpha t + (1-\alpha), \tau(i))$ for some $\tau \in S_3$ with $p_j^+ =(t_n+1, j) \in C$ for $j=1,2,3$. 
Since all three orbit lengths $n_i$ are finite, it follows that $\check f $ lifts to an automorphism on a rational surface obtained by blowing up the points in the orbits of exceptional lines $\{ (t_j+1 , \tau^{j} (i)), i = 1,2,3, j=1,\dots, n \}$.
\end{proof}

\begin{thm}\label{T:exactlysix} Let $C$ be a cubic curve with three lines joining at a single point. Suppose $n\ge 4$ such that $n$ is not divisible by both $2$ and $3$. There are exactly six quadratic birational maps $\check f_\sigma, \sigma \in S_3$ properly fixing $C$ with orbit data $n,n,n$ and a permutation $\sigma$ such that all six maps share the same base locus and the same dynamical degree $\delta_n$ where $\delta_n$ is the largest real root of $\chi_n(t)$ given in (\ref{E:ddeg}). 
\end{thm}

\begin{proof} If a quadratic birational map $\check f$ properly fixes $C$, then by Lemma \ref{L:basic} we know that $\check f$ must be basic, that is $\check f = T^- \circ J_3 \circ (T^+)^{-1}$ for some linear maps $T^\pm$. When $n$ is not divisible by both $2$ and $3$, due to Lemma \ref{L:orbitdata} one can choose the rotation $\tau$ of invariant lines based on the permutation $\sigma$ in the orbit data as follows:
\begin{itemize}
\item if $\sigma = id$ or $(i,j)$ a transpose, then $\tau = \sigma$
\item if $n\equiv 1$ (mod $3$) and $\sigma $ is a cyclic permutation, then $\tau = \sigma$
\item if $n\equiv 2$ (mod $3$) and $\sigma $ is a cyclic permutation, then $\tau = \sigma^2$
\end{itemize}
For any $\tau \in S_3$ and $j\ge 1$, $\{ \tau^j(i), i=1,2,3\} = \{1,2,3\}$. Thus the base locus \[ \{ (t_j+1, \tau^j(i)), i=1,2,3, j=1,\dots, n\} = \{ (t_j+1, i), i = 1,2,3, j=1,\dots, n\}.\]
Therefore this Theorem is the immediate consequence of Theorem \ref{T:bdk} and Proposition \ref{P:param}.
\end{proof}

\begin{rem}Let $\check f = T^- \circ J_3 \circ (T^+)^{-1}$ be a basic quadratic map in the previous Proposition \ref{P:param}. Since the automorphism on $\mathbf{P}^2$ is given by a $3\times 3$ non-singular matrix and $p_i^- = \check f \, T^+( \{ x_i=0\}) = T^- e_i$, we see the each column of $T^-$ is determined by $p_i^-$s. Similarly, $T^+$ is determined by $p_i^+$. With the condition that $\check f$ fixes the singular point (the common intersection point of three lines), one can get the explicit formula for $\check f$. The computer codes written in SageMath and Mathematica are available for download at the author's website \texttt{https://www.math.fsu.edu/~kim/publication.html}. 
\end{rem}

\section{A rational surface $X$ with $15$ blowups}\label{S:surface}
By Theorem \ref{T:exactlysix}, we see that if the common orbit length $n=5$, then all six permutations can occur in the orbit data. Let us suppose that \[ n_1=n_2=n_3 =5. \] With the $t_i$'s defined in Proposition \ref{P:param}, let us define a set of $15$ distinct points in the invariant cubic $C$. 
\begin{equation} \label{E:baselocus}
B = \{ p_{i,j}:= ( 1+ t_j, i ) \in C : i = 1,2,3, j= 1, 2, \dots, 5 \} \end{equation}
Let $X =B\ell_B( \mathbf{P}^2)$ be a rational surface obtained by blowing up $15$ distinct points in $B \subset \mathbf{P}^2$.

\begin{thm}\label{T:six}
Let $C$ be a cubic defined in (\ref{E:cubic}) and let $n=5$. Then for each $\sigma\in S_3$, there is the unique basic birational map $\check f_\sigma$ properly a cubic $C$ with orbit data $n,n,n$ and $\sigma$. Furthermore, $\check f_\sigma$ lift to an automorphism $f_\sigma : X \to X$ on a rational surface $X$ defined above. 
\end{thm}
\begin{proof}
The existence and the uniqueness of such birational map $\check f_\sigma$ is clear from Proposition \ref{P:param} and Theorem \ref{T:bdk}. From Lemma \ref{L:orbitdata}, we see that the permutation  $\tau$ of three lines must be equal to $\sigma^{-1}$. It follows that the union of the orbit of exceptional lines is given by $ \{ (t_j+1 ,  \sigma^{-j} (i)), i = 1,2,3, j=1,\dots, 5 \}$. Since  $t_j$ does not depend on $i$ and $\sigma \in S_3$ is a permutation, we see that \[\{ (t_j+1 ,  \sigma^{-j} (i)), i = 1,2,3, j=1,\dots, 5 \} = \{ (t_j+1 ,  i), i = 1,2,3, j=1,\dots, 5 \} \quad \text{for all } \sigma \in S_3.\]
And thus, for all $\sigma \in S_3$,  $\check f_\sigma$ lifts to an automorphism on $X = B\ell_B( \mathbf{P}^2)$. 
\end{proof}
Let $e_0$ be the pre-image of a generic line in $\mathbf{P}^2$ and $e_i$ be the class of the exceptional divisor over the point $(t_j+1, k) \in C$ with $i= (k-1) 5+ j$. Let $\{ e_0,e_1, \dots, e_{15}\}$ be an ordered basis of $Pic(X)$ \[ Pic(X) = \langle e_0,e_1, \dots, e_{15} \rangle. \] Recall that $C$ given in (\ref{E:cubic}) consists of three lines $L_1,L_2,$ and $L_3$ joining at one point. 
Each line $L_i$ contains five centers of blowups $(t_j+1, i), j=1,2,\dots,5$. Thus the classes of their strict transformations are given by
\begin{equation}\label{E:pic}
[C] = 3 e_0 - \sum_{i=1}^{15} e_i, \qquad \text{ and } \quad  [L_i] = e_0 - \sum_{j=1}^5 e_{j+(i-1) 5}, \ \ i =1,2,3.
\end{equation} 
Let us use $C$ for both the invariant cubic in $\mathbf{P}^2$ and its strict transformation in $X$.


\subsection{Six Quadratic Automorphisms}\label{SS:sixaut}

The coordinate functions of  those six basic maps are given by polynomials in $\mathbb{Z}(\alpha)$ where $\alpha$ is a root of the Salem factor  
\begin{equation}\label{E:salemfactor} \varphi(t) = t^4 - 2 t^3 + t^2 - 2 t+1 \end{equation} of the polynomial $\chi_5(t)$ in (\ref{E:ddeg}).
Let us set \[ \mathit{s} \ =\  \frac{ 2 \alpha^5+1}{\alpha^5-1}, \qquad \text{and } \qquad \mathit{r} \ =\  \frac{2\alpha^5+1}{2 \alpha^5-3 \alpha+1}.  \]
Let us define automorphisms on $\mathbf{P}^2$ as follows:
 \[ \mathcal{S} =\frac{1}{3}  \begin{bmatrix} \ 1&\ 1&\ 1 \\ \ \mathit{s}& -\mathit{s}&\ 0\\\ \mathit{s}&\ 0&-\mathit{s} \end{bmatrix}, \quad \mathcal{T}_{id} = \frac{1}{3}  \begin{bmatrix} \ 1&\ 1&\ 1\\ -\mathit{r}&\ \mathit{r}&\ 0\\-\mathit{r}&\ 0&\ \mathit{r} \end{bmatrix} , \quad \mathcal{T}_{(12)} = \frac{1}{3}  \begin{bmatrix} \ 1&\ 1&\ 1\\ \ \mathit{r}&-\mathit{r}&\ 0\\ \ 0&- \mathit{r}&\ \mathit{r} \end{bmatrix},\]
 \[  \mathcal{T}_{(13)} = \frac{1}{3}  \begin{bmatrix} \ 1&\ 1&\ 1\\ \ 0&\ \mathit{r}&-\mathit{r}\\\ \mathit{r}& \ 0&- \mathit{r} \end{bmatrix},  \qquad \quad \mathcal{T}_{(23)} = \frac{1}{3}  \begin{bmatrix} \ 1&\ 1&\ 1\\ -\mathit{r}&\ 0& \mathit{r}\\ - \mathit{r}&\ \mathit{r}&\ 0 \end{bmatrix},\]
  \[  \mathcal{T}_{(123)} = \frac{1}{3}  \begin{bmatrix} \ 1&\ 1&\ 1\\ \ 0&- \mathit{r}&\ \mathit{r}\\ \mathit{r}&\ -\mathit{r}& \ 0 \end{bmatrix},  \quad \text{and} \quad \mathcal{T}_{(132)} = \frac{1}{3}  \begin{bmatrix} \ 1&\ 1&\ 1\\ \ \mathit{r}&\ 0& -\mathit{r}\\ \ 0& \ \mathit{r}&- \mathit{r} \end{bmatrix}. \]
 Here we are using the the cycle notation for a permutation $\sigma \in S_n, n \in \mathbb{N}_{>0}$:  \[ \sigma = (a_{i_1}, a_{i_2}, \dots, a_{i_j}) \in S_n \ \ \Leftrightarrow\ \  \ \left\{ \begin{aligned} &\sigma(a_{i_k}) = a_{i_{k+1}},\ \ \text{for } k=1, \dots, j-1, \\
 &\sigma(a_{i_j}) = a_{i_1}, \ \ \text{and} \\
 &\sigma(a_s) = a_s,\ \  \text{if }  s \ne i_1, i_2, \dots, i_j \\ \end{aligned} \right.. \] 
  
Then for each $\sigma \in S_3$, with the Cremona involution $J_3: \mathbf{P}^2 \dasharrow \mathbf{P}^2$ we have 
 \begin{equation}\label{E:sixaut} \check f _\sigma = \mathcal{T}_\sigma \circ J_3 \circ \mathcal{S}^{-1}. \end{equation} Notice that in terms of $t_i$'s in (\ref{E:tsol}), we have $s= -1/(t_5+1)$ and $r=1/(t_1+1)$. Thus it is easy to check that each $\check f_\sigma$ is the one constructed in the subsection \ref{SS:diller}. All six birational maps lift to automorphisms $f_\sigma:X \to X$. 


\subsection{Induced Actions on $Pic(X)$}\label{SS:sixaction}
Those six birational maps $\check f_\sigma, \sigma \in \Sigma_3$  share the same base locus and satisfy (\ref{E:grouplaw}). To describe their actions on $Pic(X)$, let $\vs_\kappa$ be the Cremona involution through the vector $\ve - \ve_{5} - \ve_{10} - \ve_{15}$:
\[ \vs_\kappa \ \ : \ \ \left\{ \ \begin{aligned} & \ve_0 \mapsto 2 \ve_0 - \ve_5- \ve_{10} - \ve_{15} \\ & \ve_i \mapsto  \ve_0 - \ve_5- \ve_{10} - \ve_{15} + \ve_i, \qquad \text{if} \quad i=5.10,15 \\ & \ve_j \mapsto \ve_j, \qquad \qquad \text{ if } \quad i \ne 0,5,10,15 \\ \end{aligned} \right. \]
 With the cycle notation of permutations in $S_{15} $ we have
 
 \begin{equation}\label{E:sixaction}
 \begin{aligned}
 & f_{id}^*  = \vs_\kappa\, (5\ 4\ 3\ 2\ 1) \,(10\ 9\ 8\ 7\ 6)\, (15\ 14\ 13\ 12\ 11) \\
 & f_{(12)}^*  = \vs_\kappa\, (5\ 9\ 3\ 7\ 1\ 10\ 9\ 8\ 7\ 6)\, (15\ 14\ 13\ 12\ 11) \\
  & f_{(13)}^*  = \vs_\kappa\, (5\ 14\ 3\ 12\ 1\ 15\ 4\ 13\ 2\ 11)\, (10\ 9\ 8\ 7\ 6)\\
  & f_{(23)}^*  = \vs_\kappa\, (5\ 4\ 3\ 2\ 1) \,(10\ 14\ 8\ 12\ 6\ 15\  9\ 13\  7\ 11) \\ 
    & f_{(123)}^*  = \vs_\kappa\, (5\ 9\ 13\ 2\ 6\ 15\ 4\ 8\ 12\ 1\ 10\ 14\ 3\ 7\ 11)\\
        & f_{(132)}^*  = \vs_\kappa\, (5\ 14\ 8\ 2\ 11\ 10\ 4\ 13\ 7\ 1\ 15\ 9\ 3\ 12\ 6).\\
 \end{aligned}
 \end{equation}
 
 These actions are identified with elements in $W_{15}$. Let  $G \subset W_{15}$  be a subgroup generated by those six actions in (\ref{E:sixaction}):
  \begin{equation}\label{E:subgroupG}
 G = \langle\  f_\sigma^*, \sigma \in S_3 \ \rangle 
 \end{equation}

\section{Automorphism group $\text{Aut}(X)$ with six generators}\label{S:autgroup}
To understand the automorphism group $\text{Aut}(X)$, let us start discussing the properties of automorphisms on $X$.

\begin{prop}\label{P:allfixC}
If $\check f: \mathbf{P}^2 \to \mathbf{P}^2$ be a birational map which lifts to an automorphism $f \in \text{Aut}(X)$, then $\check f$ properly fixes the invariant cubic $C$.
\end{prop}

\begin{proof}
Notice that if $\check f$ has a point of indeterminacy in $\mathbf{P}^2$, then this point must be blown up to get an automorphism on $X$. Thus all points of indeterminacy are in the base locus, which is in the set of regular points of $C$. Since $C$ is an anti-canonical curve, so is $f(C)$, i.e., $f(C)$ is a degree $3$ curve passing through all exceptional divisors. It follows that $\check f(C)$ is a degree $3$ curve passing through all $15$ points in the base locus. The cubic $C$ consists of three lines such that each line contains exactly 5 (distinct) points in the base locus. There exists a unique such cubic. Thus we have $\check f(C) = C$. 
\end{proof}

\begin{prop}\label{P:fixC}
If $f \in \text{Aut}(X)$ then $f$ fixes the cubic $C$.
\end{prop}

\begin{proof}
Since $f$ covers a birational map $\check f$ on  $\mathbf{P}^2$, this Proposition is the immediate consequence of the previous Proposition.
\end{proof}

Also, we have
\begin{lem}\label{L:nonodal}
There is no curve $W\subset \mathbf{P}^2$ such that $W \cap C \subset B$. 
\end{lem}

\begin{proof}
Let $W$ be a curve in $\mathbf{P}^2$ and suppose $W \cap C \subset B$. By the group law, there are non-negative integers $m_1, \dots, m_5$ such that 
\begin{equation}\label{E:sum} \sum_{i=1}^5 m_i(1+t_i) = \frac{1}{1+ 2 \alpha^5} \, \left(\sum_{i=1}^5  m_i ( 1+ 2 \alpha^5 - 3 \alpha^i)\right) = 0 \end{equation}
where $\alpha$ is a root of the Salem polynomial $t^4-2 t^3+t^2-2 t +1$. 
Since $1+ 2\alpha^5 \ne 0$, (\ref{E:sum}) satisfies if and only if 
\[ \begin{aligned} \sum_{i=1}^5 &  m_i ( 1+ 2 \alpha^5 - 3 \alpha^i) \\ & = ( 2 \sum m_i - 3 m_5) \alpha^5 - 3 m_4 \alpha^4 - 3 m_3 \alpha^3 - 3 m_2 \alpha^2 - 3 m_1 \alpha + \sum m_i \ = \ 0\end{aligned}\]
Since the minimal polynomial of $\alpha$ is of degree $4$, there are integers $A, B$ such that
\[\sum_{i=1}^5   m_i ( 1+ 2 \alpha^5 - 3 \alpha^i) = (\alpha^4 - 2 \alpha^3 +\alpha^2 - 2 \alpha +1) ( A \alpha + B). \]
By comparing the coefficients, we get $ m_1=m_3=m_5=-m_2=-m_4$. Since $m_i$ are non-negative integers, the only possible solution is $m_i = 0 $ for all $i$. 
\end{proof}

\begin{prop}\label{P:nofix}
If $f \in \text{Aut}(X)$ is not linear, then the cubic $C$ does not contain a curve of fixed points.
\end{prop}

\begin{proof}
Suppose one component, say $L_1$, is a curve of fixed points. Since $L_1$ intersects exceptional curves $E_j, j=1, \dots, 5$ over $p_{1,j}$, for each $j=1, \dots 5$ either $E_j$ is fixed or $p_{1,j}$ is a point of indeterminacy. If all $E_j$'s are fixed, then we have $f^* e_i = e_i$ for all $i = 1, \dots, 5$ and $f^*e_0 = f^* ([L_1]- e_1-\cdots e_5)  = e_0$. It follows that $f$ is linear. Suppose one of $p_{1,j}$ is a point of indeterminacy. Since $L_1$ is a curve of fixed points, the exceptional curve $E_j$ maps to a curve $ fE_j$ passing through $L_1 \cap E_j$. Since $f$ preserves the intersection number, $f E_j$ can not intersect $C$. It follows that the $f E_j$ descends to a curve intersecting $C$ only at the base locus $B$. This is a contradiction by Lemma \ref{L:nonodal}.
\end{proof}

\subsection{Dynamical degrees and Determinants}

Note that $C\subset X$ is an anticanonical curve. If $f \in \text{Aut}(X)$ preserves the invariant cubic $C$, then $f^* \eta = d(f) \,\eta$ where $\eta$ is the unique (up to constant multiple) meromorphic form on $X$ with a simple pole along $C$. We call $d(f)$ the \textit{determinant} of $f$ as defined in \cite{McMullen:2007}. If a birational map $\check f$ properly fixing $C$ lifts to $f$ then the restriction $\check f|_C\, (x,i) = (d(f) x + b, \tau(i) )$ for a constant $b$ and a permutation $\tau$. The determinant $d(f)$ is, in fact, the determinant of the differential $\text{det} Df_p$ at the fixed point $p \notin C$. If the determinant $d(f)$ is not a root of unity, then $d(f)$ is a Galois conjugate of the dynamical degree $\delta(f)$. And we have

\begin{thm}[\cite{McMullen:2007,Uehara:2010}]
If $f \in \text{Aut}(X)$ is an automorphism on $X$, then the determinant $d(f)$ of $f$ is an eigenvalue of $f^*|_{Pic(X)}$. 
\end{thm}

Clearly, we have 

\begin{lem}
If $f:X \to X$ is an automorphism on $X$, then we have $ d(f^{-1}) = \frac{1}{d(f)}. $
\end{lem} 

\begin{lem}\label{L:det_prod}
If $f,g \in \text{Aut}(X)$ then the determinant of $f \circ g \in \text{Aut}(X)$ is the product $d(f) d(g)$ of determinants of $f$ and $g$. 
\end{lem}

\begin{proof}
From Proposition \ref{P:fixC}, we see that $f, g$ and $f\circ g$ fix the cubic $C$. Thus the restriction maps satisfy $(f\circ g)|_C = f|_C \circ g|_C$ and $d(f\circ g) = d(f) d(g)$. 
\end{proof}

%

Since the characteristic polynomial of $f^*|_{Pic(X)}$ is given by a product of a Salem polynomial and cyclotomic polynomials, it follows that 
\begin{lem}\label{L:dyn_prod}
Suppose $f$ and $g$ are automorphisms on $X$ with dynamical degrees $\delta(f)\ge \delta(g) >1$. If determinants of $f$ and $g$ are $d(f)$ and $d(g)$ respectively, then the dynamical degree $\delta (f\circ g)$ of $f\circ g$ is either $\delta(f) \delta(g)$, or $\delta(f)/\delta(g)$. Furthermore, by considering $g^{-1}$ we can get both dynamical degree $\delta(f) \delta(g)$, and $\delta(f)/\delta(g)$.
\end{lem}

\begin{proof}
Since the determinant $d(f)$ is a Galois conjugate of the dynamical degree $\delta (f)$ and the dynamical degree is a Salem number, the minimal polynomial $P_f(t)$ of $d(f)$ is given by a reciprocal poly $P_f(t) = \prod (t- \alpha_i)$ with all but two $\alpha_i$ are non-real complex numbers with modulus $1$. Similarly, the minimal polynomial of $d(g)$ is given by $P_g(t) = \prod (t-\beta_j)$ with only two $\beta_j$ are real. 
The minimal polynomial $P_{f\circ g}$ of the product $d(f) d(g)$ must divide the polynomial $\prod( t- \alpha_i \beta_j)$ \[ P_{f\circ g}(t) | \prod_{i,j}( t- \alpha_i \beta_j). \]Since only four real numbers in $\{ \alpha_i \beta_j \}$ are $\delta(f) \delta(g)$, $\delta(f)/\delta(g)$, and their reciprocals. By Lemma \ref{L:det_prod}, we know that the product $d(f) d(g)$ must be a Galois conjugate of the dynamical degree $\delta(f\circ g)$ of the composition. Since the dynamical degree is a real number $\ge 1$, we see that the dynamical degree $\delta (f\circ g)$ of $f\circ g$ is either $\delta(f) \delta(g)$ or $\delta(f)/\delta(g)$.  
\vspace{1ex}
If $\delta(f \circ g) = \delta(f) \delta(g)$ then both $f^*$ and $g^*$ have the same eigenvector $v_+$ corresponding the largest eigenvalues $\delta(f), \delta(g)$. It follows that $f$ and $g^{-1}$ will not share the same eigenvector corresponding to the largest eigenvalues, and thus $\delta (f \circ g^{-1}) \ne \delta(f) \delta(g^{-1})$. Since the only possible options are $\delta(f) \delta(g^{-1})$ and $\delta(f)/\delta(g^{-1})$ and $\delta(g) = \delta(g^{-1})$, we have that $\delta (f \circ g^{-1}) = \delta(f)/\delta(g)$. 
\end{proof}

\begin{thm}\label{T:dpower}
If $f \in \text{Aut}(X)$ is an automorphism on $X$, then the dynamical degree $\delta(f)$ of $f$ is $\delta^n$ for some $n \ge 0$ where $\delta$ is the largest real root of the Salem polynomial \[ \chi_5(t) = t^4-2 t^3+t^2-2 t+1 \]
\end{thm}
\begin{proof}
Recall that there are six automorphisms $f_\sigma:X \to X, \sigma \in S_3$ constructed in Theorem \ref{T:six} and $\delta(f_\sigma) = \delta$. Suppose there is an automorphism $g \in \text{Aut}(X)$ such that $\delta(g) \ne \delta^n$ for all $n\ge 0$. The composition $f_\sigma \circ g \in \text{Aut}(X)$. Furthermore due to Lemma \ref{L:dyn_prod}, the dynamical degree $\delta(f_\sigma \circ g)$ is either $\delta \cdot \delta(g)$, $\delta / \delta(g)$, or $\delta(g)/\delta$. If there is no Salem number $\tau$ such that $\delta = \tau^p$ and $\delta(g) = \tau^q$ for some $p,q \ge 1$, then none of them is Salem number due to Proposition \ref{P:prodSalem}. This is a contradiction since the dynamical degree of an automorphism of a rational surface is either $1$ or a Salem number. 

Now suppose there is a Salem number $\tau$ such that $\delta = \tau^p$ and $\delta(g) = \tau^q$ for some $p,q \ge 1$. If $p=1$, then $ \delta(g) = \delta^q$. If $p>1$, applying the Lemma \ref{L:dyn_prod} repeatedly, we get an automorphism $h \in \text{Aut}(X)$ such that $\delta(h) = \tau$. Note that since $h$ is an automorphism on $X$, $h^* \in W_{15}$ and thus $\delta(h) \ge \lambda_\star$, the Lehmer's number \cite{McMullen:2002}. Direct computation shows that $\delta^{1/4} \approx 1.17145 < \lambda_\star$ and the minimal polynomials for $\delta ^{1/2}$ and $\delta^{1/3}$ have two or more real roots bigger than $1$. 
It follows that $p =1$ and therefore $\delta(g) = \delta^n$ for some $n\ge 0$. 
\end{proof}

%
%
%

\subsection{The Automorphism group} 

\begin{prop}\label{P:identityonPic}
Suppose $f\in \text{Aut}(X)$. If the induced action $f^*: Pic(X) \to Pic(X)$ is the identity, then $f$ itself is an identity map. \[ f^* \ =\ Id \quad \Leftrightarrow \quad f\ =\ Id.\] 
\end{prop}

\begin{proof}
If $f^* = Id$, then the corresponding birational map $\check f$ on $\mathbf{P}^2$ must fix all $15$ distinct base points and the three lines joining at one point. Since each line contains $5$ base points, those three lines are, in fact, lines of fixed points for $\check f$. Thus, we have $\check f = Id$ and $f=Id$.
\end{proof}

%
%

\begin{prop}\label{P:tt} Let $\delta$ is the largest real root of $\varphi(t)$ defined in (\ref{E:salemfactor}). If an automorphism $f$ on $X$ has the dynamical degree $\delta(f) = \delta$ then 
\[ f \in \{ f_\sigma, f^{-1}_\sigma: \sigma \in S_3\} \]
where $f_\sigma$ is the lift $\check f_\sigma$ defined in (\ref{E:sixaut}) \end{prop}

\begin{proof}
Suppose $f\in \text{Aut}(X)$ is an automorphism with $\delta(f) = \delta>1$. Then the determinant $d(f)$ is a root of $\varphi(t)$. Due to the construction, for each $\sigma \in S_3$, automorphisms $f_\sigma$ have the same determinant. Furthermore, $\varphi(t)$ has two real roots $\delta, 1/\delta$, and two non-real roots $\omega, \bar \omega$ on the unit circle. We have two cases:
\begin{itemize}
\item If $d(f_\sigma)$ is real, then $d(f)$ has to be real. Otherwise, we have $d(f \circ f_\sigma) = d(f) \cdot d(f)$, which is non-real outside unit circle. It follows that the determinant $d(f \circ f_\sigma) $ of composition is not a Galois conjugate of a Salem number, and thus $f \circ f_\sigma$ is not an automorphism, which is a contradiction. Thus the determinant $d(f)$ has to be real, that is, $d(f)$ is either $\delta$ or $1/\delta$. 
\item Similarly, if $d(f_\sigma)$ is non-real, then $d(f)$ is either $\omega$ or $\bar \omega$.
\end{itemize}
Either case, we have either $d(f) = d(f_\sigma)$ or $d(f) = 1/d(f_\sigma)$ for all $\sigma \in S_3$. It follows that there is $\sigma \in S_3$ such that the composition $f\circ f_\sigma$ (or $f \circ f_\sigma^{-1}$) has determinant $1$. It follows that there is a positive integer $n$ such that $(f\circ f_\sigma)^n = Id$. If the restriction map $f\circ f_\sigma|C : t \mapsto t+\beta$, then $t+n \beta = t$. It follows that $\beta=0$ and thus $C$ is a curve of fixed points. By Proposition \ref{P:nofix}, we see that $f \circ f_\sigma$ is linear and thus $(f\circ f\sigma) \circ f_\sigma^{-1}$ is quadratic. 
The statement follows from Theorem \ref{T:exactlysix}
\end{proof}

%
%

\begin{lem}\label{L:dihedral3}
Let us set six linear maps \begin{equation}\label{E:linear} L _\sigma := f_\sigma \circ f_{id}^{-1} \qquad \text{ for } \sigma \in S_3. \end{equation}
Let $\mathcal{L}$ be the group generated by six linear maps above, then $\mathcal{L}$ is equivalent to the dihedral group $D_3$
\[ \mathcal{L} = \langle L_{(12)}, L_{(13)}, L_{(23)}, L_{(123)}, L_{(132)} \rangle \ \cong D_3. \]
\end{lem}

\begin{proof}
The the formulae for $\check f_\sigma = \mathcal{T}_\sigma \circ J_3 \circ \mathcal{S}^{-1}$ in (\ref{E:sixaut}), it is easy to see $L_\sigma$ is a linear. In fact we have \[ L_{id} = L_{(12)}^2 = L_{(23)}^2 = L_{(13)}^2 = L_{(123)}^3=L_{(132)}^3=Id, \quad \text{and} \quad L_{(123)}^{-1}=L_{(132)}. \]
\end{proof}

The linear maps (\ref{E:linear}) defined in the previous Proposition permute three concurrent lines. If $\sigma \in S_3$, $L_\sigma$ maps the line $L_{\sigma(i)}$ to the line $L_i$.

\begin{prop}\label{P:linearaut}
Suppose $f \in \text{Aut}(X)$ is an automorphism on $X$. If the dynamical degree of $f$, $\delta(f) = 1$, then $f \in \mathcal{L} $
\end{prop}
\begin{proof}
Let $P(t) = \prod (t - \alpha_i)$ be the minimal polynomial of the dynamical degree $\delta>1$ of $f_\sigma$. 
If $\delta(f) = 1 $, then the determinant of $f$ is a root of unity. Thus, the minimal polynomial of $ d(f \circ f_\sigma)=d(f) d(f_\sigma)$ divides $\prod (t- \alpha_i \beta_j)$ where $\beta_j$'s are roots of unity. For any root of unity $\beta$, Since $\alpha_i \beta$ is not a root of unity and $|\alpha_i \beta|=|\alpha_i|$. Thus we see that the dynamical degree of $f \circ f_\sigma = \delta$. On the other hand, the determinant $ d(f \circ f_\sigma)=d(f) d(f_\sigma)$ is a Galois conjugate of $\delta$ and thus $d(f) =1$. With the same argument in the proof of Proposition \ref{P:tt} $\check f$ is linear and thus $\check f \circ \check f_\sigma $ is a quadratic map with dynamical degree $\delta$. Thus, by Theorem \ref{T:exactlysix}, we have $f \circ f_\sigma = f_\tau$ for some $\tau \in S_3$. Thus \[ f= f_\tau \circ f_\sigma^{-1} = L_\tau \circ L_\sigma^{-1} \in \mathcal{L}. \]
\end{proof}

\begin{thm}\label{T:autgroup} Let $X$ be a rational surface defined in the beginning of Section \ref{S:surface} and let $f_{id}$, $f_{(12)}$, $f_{(13)},f_{(23)},f_{(123)},$ and $f_{(132)}$ be six quadratic automorphisms on $X$ which are induced from birational maps defined in (\ref{E:sixaut}). The automorphism group $\text{Aut}(X)$ is generated by $6$ quadratic automorphisms properly fixing three lines meeting at a single point and \[ \text{Aut}(X) = \langle f_{id},f_{(12)},f_{(13)},f_{(23)}, f_{(123)}, f_{(132)} \rangle  \ \cong \ D_3 \rtimes \mathbb{Z} \]
\end{thm}
\begin{proof}
Direct computation using formulae in (\ref{E:sixaut}) shows that $L_\sigma \circ f_{id} = f_{id} \circ L_\sigma$. Since $f_\sigma = L_\sigma \circ f_{id}$, it is clear that the second part $\langle f_{id},f_{(12)},f_{(13)},f_{(23)}, f_{(123)}, f_{(132)} \rangle  \ \cong \ D_3 \rtimes \mathbb{Z} $. Now suppose $f \in \text{Aut}(X)$. By Theorem \ref{T:dpower}, the dynamical degree $\delta(f) = \delta^n$ for some $n\ge 0$. Applying Lemma \ref{L:dyn_prod} repeatedly and using Proposition \ref{P:linearaut}, we have the desired result. \end{proof}

\begin{proof}[Proof of Theorem A]
The assertion of Theorem $A$ is the immediate consequence of the previous theorem \ref{T:autgroup}.
\end{proof}

\section{The Subgroup of $W_{15}$}\label{S:wsub}
\begin{thm}\label{T:Wsubgroup} Let $X = B\ell_B( \mathbf{P}^2)$ be a rational surface obtained by blowing up a set $B$ of $15$ distinct points defined in Equation \ref{E:baselocus} and let $G$ be a subgroup of the Coxeter group $W_{15}$ defined in (\ref{E:subgroupG}). Then $G$ is realized by automorphisms on $X$. In fact, the subgroup $G$ of $W_{15}$ is isomorphic to the automorphism group $\text{Aut}(X)$. 

\[ G = \langle f^*_{id},f^*_{(12)},f^*_{(13)},f^*_{(23)},  f^*_{(123)}, f^*_{(132)}  \rangle   \ \cong \ \text{Aut}(X) \]
\end{thm}
\begin{proof} The map sending  $f \in \text{Aut}(X)$ to the actions $f^*$ on $Pic(X)$ is a homomorphism. This map is, in fact, one-to-one due to Proposition \ref{P:identityonPic}. It follows that $G \cong \text{Aut}(X)$.

\end{proof}

\begin{proof}[Proof of Theorem B]
The statement of Theorem $B$ is clear from the previous Theorem \ref{T:Wsubgroup}.
\end{proof}

Also, we have other subgroups that are realized by automorphisms on rational surfaces. The construction is essentially the same as the subgroup $G \in W_{15}$ in this section and the rational surface $X$ in the section \ref{S:autgroup}. We will list them below. The following automorphism groups are generated by quadratic surface automorphisms fixing concurrent lines. 

\subsection{Other subgroups} Let $X_{3n}$ be a rational surface obtained by blowing a set of finite points on a cubic $C$ given by three concurrent lines  \[ B_n = \{ p_{i,j} := (1+ t_j,i) \in C : i = 1,2,3, j=1, \dots, n\} \] 
\begin{enumerate}
\item If $n\ge 5$ is odd and $n\equiv 0$ $( \text{mod }3)$, then $\text{Aut}(X_{3n}) \cong (\mathbb{Z}/2\mathbb{Z})^2 \rtimes \mathbb{Z}$.
\item If $n\ge 4$ is even and $n\not\equiv 0$ $( \text{mod }3)$, then $\text{Aut}(X_{3n}) \cong (\mathbb{Z}/3\mathbb{Z}) \rtimes \mathbb{Z}$.
\item If $n\ge 4$ is even and $n\equiv 0$ $( \text{mod }3)$, then $\text{Aut}(X_{3n}) \cong \mathbb{Z}$.
\end{enumerate} 



 \end{document}